\documentclass[reqno,a4paper,11pt]{amsart}
\usepackage[all,poly]{xy}
\usepackage{amsfonts}
\usepackage[mathcal]{eucal}
\usepackage{amssymb}
\usepackage{amsmath}
\usepackage{mathrsfs}
\usepackage{color}
\usepackage[pagebackref,colorlinks]{hyperref}
\usepackage{enumerate}
\usepackage{hhline}
\usepackage{graphics}
\usepackage{float}
\usepackage{enumerate}
     
\theoremstyle{plain}
\newtheorem {lemma}{Lemma}
\newtheorem {replacement lemma}[lemma]{Replacement Lemma}

\newtheorem {theorem}[lemma]{Theorem}
\newtheorem {FD theorem}[lemma]{Finite Dimension Theorem}
\newtheorem {corollary}[lemma]{Corollary}

\theoremstyle{definition}
\newtheorem{definition}[lemma]{Definition}

\newtheorem{remark}[lemma]{Remark}
\newtheorem {example}[lemma]{Example}

\newcommand{\R}{\mathbb{R}}
\newcommand{\N}{\mathbb{N}}
\newcommand{\X}{\langle X \rangle}
\def\Oo{\mathcal{O}}
 
\newcommand{\GKdim}{\operatorname{GKdim}}
\newcommand{\nod}{\operatorname{nod}}
\newcommand{\reg}{\operatorname{reg}}
\newcommand{\RWF}{\operatorname{RWF}}
\newcommand{\CP}{\operatorname{CP}}
\newcommand{\st}{\operatorname{st}}

\title[Weighted Leavitt path algebras of finite GK dimension]{Weighted Leavitt path algebras of finite Gelfand-Kirillov dimension}

\author{Raimund Preusser}
\address{Department of Mathematics,
University of Brasilia, Brazil}
\email{raimund.preusser@gmx.de}
\date{}


\begin{document}

\begin{abstract} We determine the Gelfand-Kirillov dimension of a weighted Leavitt path algebra $L_K(E,w)$ where $K$ is a field and $(E,w)$ a finite weighted graph. Further we show that a finite-dimensional weighted Leavitt path algebra over a field $K$ is isomorphic to a finite product of matrix rings over $K$. 
\end{abstract}

\maketitle

\section{Introduction}
In a series of papers William Leavitt studied algebras that are now denoted by $L_K(n,n+k)$ and have been coined Leavitt algebras. Let $X=(x_{ij})$ and $Y=(y_{ji})$ be $(n+k)\times n$ and $n\times (n+k)$ matrices consisting of symbols $x_{ij}$ and 
$y_{ji}$, respectively. Then for a field $K$, $L_K(n,n+k)$ is a $K$-algebra generated by all $x_{ij}$ and $y_{ji}$ subject to the relations $XY=I_{n+k}$ and $YX=I_n$. In~\cite[p.190]{vitt56} Leavitt studied these algebras for  $n=2$ and $k=1$, in~\cite[p.322]{vitt57} for any $n\geq 2 $ and $k=1$ and finally in~\cite[p.130]{vitt62} for arbitrary $n$ and $k$. 
 
 Leavitt path algebras (Lpas) were introduced a decade ago~\cite{aap05,Ara_Moreno_Pardo}, associating a $K$-algebra to a directed graph. For a graph with one vertex and $k+1$ loops, it recovers the Leavitt algebra $L_K(1,k+1)$.  The definition and the development of the theory were  inspired on the one hand by Leavitt's construction of  $L_K(1,k+1)$ and on the other hand by Cuntz algebras $\Oo_n$~\cite{cuntz1} and  Cuntz-Krieger algebras in $C^*$-algebra theory~\cite{raeburn}.  The Cuntz algebras and later Cuntz-Krieger type $C^*$-algebras revolutionised $C^*$-theory, leading ultimately to the astounding
Kirchberg-Phillips classification theorem~\cite{phillips}. In the last decade the Lpas have created the same type of stir in the algebraic community. The development of Lpas and its interaction with graph $C^*$-algebras have been well-documented in several publications and we refer the reader to~\cite{abrams-ara-molina} and the references therein. 
 
 Since their introductions, there have been several attempts to introduce a generalisation of Lpas which would cover the algebras $L_K(n,n+k)$ for any $n\geq 1$, as well. Ara and Goodearl's Leavitt path algebras of separated graphs were introduced in~\cite{aragoodearl} which gives $L_K(n,n+k)$ as a corner ring of some separated graphs. The weighted Leavitt path algebras (wLpas) were introduced by R. Hazrat in \cite{hazrat13}, which gives $L_K(n,n+k)$ for a weighted graph with one vertex and $n+k$ loops of weight $n$. If the weights of all the edges are $1$ (i.e., the graph is unweighted), then the wLpas reduce to the usual Lpas. 

In \cite{zel12} linear bases for Lpas were obtained and used to determine the Gelfand-Kirillov dimension of an Leavitt path algebra $L_K(E)$ where $K$ is a field and $E$ a finite directed graph. In \cite{hazrat-preusser} linear bases for wLpas were obtained, generalising the basis result for Lpas given in \cite{zel12}. These bases were used to classify the wLpas which are domains, simple and graded simple rings. 
In this note we use them to determine the Gelfand-Kirillov dimension of a weighted Leavitt path algebra $L_K(E,w)$ where $K$ is a field and $(E,w)$ a finite weighted graph. Further we show that a finite-dimensional weighted Leavitt path algebra over a field $K$ is isomorphic to a finite product of matrix rings over $K$. 

The rest of the paper is organised as follows. In Section 2 we recall some basic facts about wLpas. In Section 3 we prove our first main result, Theorem \ref{thmm}, which gives the Gelfand-Kirillov dimension of a weighted Leavitt path algebra $L_K(E,w)$ where $K$ is a field and $(E,w)$ a finite weighted graph. In Section 4 we compute the Gelfand-Kirillov dimension of some concrete examples of wLpas. In Section 5 we prove our second main result, the Finite Dimension Theorem \ref{thmm2}.
 
\section{Weighted Leavitt path algebras}

Throughout this section $R$ denotes an associative, unital ring.

\begin{definition}[\sc Directed graph]\label{defdg}
A {\it directed graph} is a quadruple $E=(E^0,E^1,s,r)$ where $E^0$ and $E^1$ are sets and $s,r:E^1\rightarrow E^0$ maps. The elements of $E^0$ are called {\it vertices} and the elements of $E^1$ {\it edges}. If $e$ is an edge, then $s(e)$ is called its {\it source} and $r(e)$ its {\it range}. $E$ is called {\it row-finite} if $s^{-1}(v)$ is a finite set for any vertex $v$ and {\it finite} if $E^0$ and $E^1$ are finite sets.
\end{definition}

\begin{definition}[{\sc Double graph of a directed graph}]\label{defddg}
Let $E$ be a directed graph. The directed graph $E_d=(E_d^0, E_d^1, s_d, r_d)$ where $E_d^0=E^0$, $ E_d^1=E^1\cup (E^1)^*$ where $(E^1)^*=\{e^*\mid e\in E^1\}$,
\[s_d(e)=s(e),~r_d(e)=r(e),~s_d(e^*)=r(e) \text{ and }r_d(e^*)=s(e)\text{ for any }e\in E^1\]
is called the {\it double graph of $E$}. We sometimes refer to the edges in the graph $E$ as {\it real edges} and the additional edges in $E_d$ (i.e. the elements of $(E^1)^*$) as {\it ghost edges}. 
\end{definition}

\begin{definition}[{\sc Path}]\label{deffdp}
Let $E$ be a directed graph. A {\it path} is a nonempty word $p=x_1\dots x_n$ over the alphabet $E^0\cup E^1$ such that either $x_i\in E^1~(i=1,\dots,n)$ and $r(x_i)=s(x_{i+1})~(i=1,\dots,n-1)$ or $n=1$ and $x_1\in E^0$. By definition, the {\it length} $|p|$ of $p$ is $n$ in the first case and $0$ in the latter case. We set $s(p):=s(x_1)$ and $r(p):=r(x_n)$ (here we use the convention $s(v)=v=r(v)$ for any $v\in E^0$). A {\it closed path} is a path $p$ such that $|p|\neq 0$ and $s(p)=r(p)$. A {\it cyclic path} is a closed path $p=x_1\dots x_n$  such that $s(x_i)\neq s(x_j)$ for any $i\neq j$.
\end{definition}

\begin{definition}[\sc Path algebra]\label{defpa}
Let $E$ be a directed graph. The quotient $R\langle E^0\cup E^1\rangle/I $ of the free $R$-ring $R\langle E^0\cup E^1\rangle$ generated by $E^0\cup E^1$ and the ideal $I$ of $R\langle E^0\cup E^1\rangle$ generated by the relations
\begin{enumerate}[(i)]
\item $vw=\delta_{vw}v$ for any $v,w\in E^0$ and
\item $s(e)e=e=er(e)$ for any $e\in E^1$
\end{enumerate}
is called the {\it path algebra of $E$} and is denoted by $P_R(E)$.
\end{definition}

\begin{remark}
The paths in $E$ form a basis for the path algebra $P_R(E)$.
\end{remark}

\begin{definition}[{\sc Weighted graph}]\label{defsg}
A {\it weighted graph} is a pair $(E,w)$ where $E$ is a directed graph and $w:E^1\rightarrow \N=\{1,2,\dots\}$ is a map. If $e\in E^1$, then $w(e)$ is called the {\it weight} of $e$. A weighted graph $(E,w)$ is called {\it row-finite} (resp. {\it finite}) if $E$ is row-finite (resp. finite). In this article all weighted graphs are assumed to be row-finite and to have at least one vertex.
\end{definition}

\begin{remark}
Let $(E,w)$ be a weighted graph. In \cite{hazrat13} and \cite{hazrat-preusser}, $E^1$ is denoted by $E^{\st}$. What is denoted by $E^1$ in \cite{hazrat13} and \cite{hazrat-preusser} is denoted by $\hat E^1$ in this article (see the next definition).
\end{remark}

\begin{definition}[{\sc Directed graph associated to a weighted graph}]
Let $(E,w)$ be a weighted graph. The directed graph $\hat E=(\hat E^0, \hat E^1, \hat s, \hat r)$ where $\hat E^0=E^0$, $\hat E^1:=\{e_1,\dots,e_{w(e)}\mid e\in E^1\}$, $\hat s(e_i)=s(e)$ and $\hat r(e_i)=r(e)$ is called the {\it directed graph associated to $(E,w)$}. We sometimes refer to the edges in the weighted graph $(E,w)$ as {\it structured edges} to distinguish them from the edges in the associated directed graph $\hat E$.
\end{definition}

Until the end of this section $(E,w)$ denotes a weighted graph. A vertex $v\in E^0$ is called a {\it sink} if $s^{-1}(v)=\emptyset$ and {\it regular} otherwise. The set of all regular vertices is denoted by $E^0_{\reg}$. For a $v\in E^0_{\reg}$ we set $w(v):=\max\{w(e)\mid e\in s^{-1}(v)\}$. $\hat E_d$ denotes the double graph of the directed graph $\hat E$ associated to $(E,w)$.

\begin{definition} [{\sc Weighted Leavitt path algebra}]\label{def3}
The quotient $P_R(\hat E_d)/I$ of the path algebra $P_R(\hat E_d)$ and the ideal $I$ of $P_R(\hat E_d)$ generated by the relations
\begin{enumerate}[(i)]
\item $\sum\limits_{e\in s^{-1}(v)}e_ie_j^*= \delta_{ij}v$ for all $v\in E^0_{\reg}$ and $1\leq i, j\leq w(v)$ and
\medskip 

\item $\sum\limits_{1\leq i\leq w(v)}e_i^*f_i= \delta_{ef}r(e)$ for all $v\in E^0_{\reg}$ and $e,f\in s^{-1}(v)$
\end{enumerate}
is called {\it weighted Leavitt path algebra of $(E,w)$} and is denoted by $L_R(E,w)$. In relations (i) and (ii), we set $e_i$ and $e_i^*$ zero whenever $i > w(e)$. 
\end{definition}


\begin{example}\label{wlpapp}
Let $K$ be a field. It is easy to see that the wLpa of a weighted graph consisting of one vertex and $n+k$ loops of weight $n$ is isomorphic to the Leavitt algebra $L_K(n,n+k)$, for details see \cite[Example 4]{hazrat-preusser}. 
\end{example}

\begin{example}\label{exex1}
If $w(e)=1$ for all $e \in E^{1}$, then $L_R(E,w)$ is isomorphic to the usual Leavitt path algebra $L_R(E)$. 
\end{example}

We call a path in the double graph $\hat E_d$ a {\it d-path}. While the d-paths form a basis for the path algebra $P_R(\hat E_d)$, a basis for the weighted Leavitt path algebra $L_R(E,w)$ is formed by the nod-paths, which we will define in the next definition.

For any $v\in E^0_{\reg}$ fix an $e^{v}\in s^{-1}(v)$ such that $w(e^{v})=w(v)$. The words
\[e^v_i(e^v_j)^*~(v\in E^0_{\reg},1\leq i,j\leq w(v))\text{ and }e^*_1f_1~(v\in E^0_{\reg},e,f\in s^{-1}(v))\]
over the alphabet $\hat E_d^1$ are called {\it forbidden}. If $A=x_1\dots x_n$ is a word over some alphabet, then we call the words $x_i\dots x_j~(1\leq i\leq j\leq n)$ {\it subwords of $A$}.   

\begin{definition}[{\sc Nod-path}]\label{deffnod}
A {\it normal d-path} or {\it nod-path} is a d-path such that none of its subwords is forbidden. 
\end{definition}

\begin{theorem}[Hazrat, Preusser, 2017] \label{thmhp}
The nod-paths form a basis for $L_R(E,w)$.
\end{theorem} 
\begin{proof}
See \cite[Theorem 16]{hazrat-preusser}
\end{proof}

\section{Weighted Leavitt path algebras of polynomial growth}
First we want to recall some general facts on the growth of algebras. Let $K$ be a field and $A$ an $K$-algebra (not necessarily unital), which is generated by a finite-dimensional subspace $V$. For $n\geq 1$ let $V^n$ denote the span of the set $\{v_1\dots v_k\mid k\leq n, v_1,\dots,v_k\in V\}$. Then $V =V^1\subseteq V^2\subseteq \dots$, $A =\bigcup\limits_{n\geq 1}V^n$ and $d_V(n):=\dim V^n<\infty$. Given functions $f, g$ from the positive integers $\N$ to the positive real numbers $\R^+$, we write $f\preccurlyeq g$ if there is a $c\in\N$ such that $f(n)\leq cg(cn)$ for all $n$. If $f\preccurlyeq g$ and $g\preccurlyeq f$, then the functions $f, g$ are called {\it asymptotically equivalent} and we write $f\sim g$. If $W$ is another finite-dimensional subspace that generates $A$, then $d_V\sim d_W$. The {\it Gelfand-Kirillov dimension} or {\it GK dimension} of $A$ is defined as
\[\GKdim A := \limsup\limits_{n\rightarrow \infty}\log_nd_V(n).\]
The definition of the GK dimension does not depend on the choice of the finite-dimensional generating space $V$. If $d_V\preccurlyeq n^m$ for some $m\in \N$, then $A$ is said to have {\it polynomial growth} and we have $\GKdim A \leq m$. If $d_V\sim a^n$ for some real number $a>1$, then $A$ is said to have {\it exponential growth} and we have $\GKdim A =\infty$. 

Until right after the proof of Theorem \ref{thmm}, $K$ denotes a field and $(E,w)$ a finite weighted graph. Further $V$ denotes the finite-dimensional subspace of $L_K(E,w)$ spanned by $\hat E_d^0\cup \hat E_d^1$ (i.e. spanned by the vertices, real edges and ghost edges).

If $X$ is a set, we denote by $\X$ the set of all words over $X$ including the empty word. Together with juxtaposition $\X$ is a monoid. If $A,B\in \X$, then we write $A|B$ if there is a $C\in\X$ such that $AC=B$.

\begin{definition}\label{deffnodc}
Let $p$ and $q$ be nod-paths. If there is a nod-path $o$ such that $p\!\!\not| \, o$ and $poq$ is a nod-path, then we write $p\overset{\nod}{\Longrightarrow} q$. If $pq$ is a nod-path or $p\overset{\nod}{\Longrightarrow} q$, then we write $p\Longrightarrow q$.
\end{definition}


\begin{definition}[{\sc Nod$^2$-path, quasi-cycle}]\label{deffnod2}
A {\it nod$^2$-path} is a nod-path $p$ such that $p^2$ is a nod-path. A {\it quasi-cycle} is a nod$^2$-path $p$ such that none of the subwords of $p^2$ of length $<|p|$ is a nod$^2$-path. A quasi-cycle $p$ is called {\it selfconnected} if $p\overset{\nod}{\Longrightarrow}p$.
\end{definition}
 
\begin{remark}\label{remnod2}
$~$\vspace{-0.2cm}
\begin{enumerate}[(a)]
\item
It is easy to see that if $p=x_1\dots x_n$ is a quasi-cycle, then $x_i\neq x_j$ for all $i\neq j$ (otherwise there would be a subword of $p^2$ of length $<|p|$ that is a nod$^2$-path). It follows that there is only a finite number of quasi-cycles since $E$ is finite.
\item
Let $p=x_1\dots x_n$ be a quasi-cycle and $\pi\in S^n$ an $n$-cycle. Then $q:=x_{\pi(1)}\dots x_{\pi(n)}$ is a quasi-cycle and we write $p\approx q$. Clearly $\approx$ is an equivalence relation on the set of all quasi-cycles.
\item 
Let $p=x_1\dots x_n$ be a quasi-cycle. Then $p^*:=x_n^*\dots x_1^*$ is a quasi-cycle.
\end{enumerate}
\end{remark}

\begin{example}\label{exqu}
Suppose $(E,w)$ is the weighted graph 
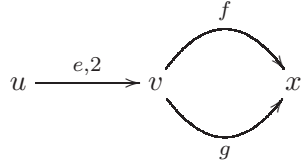
\[
\xymatrix@C+15pt{
u\ar[r]^{e,2}&  v\ar@/^1.7pc/[r]^{f}\ar@/_1.7pc/[r]_{g}  & x  
}
\]
(here $e$ has weight $2$ and $f$ and $g$ have weight $1$). Then the associated directed graph $\hat E$ and its double graph $\hat E_d$ are
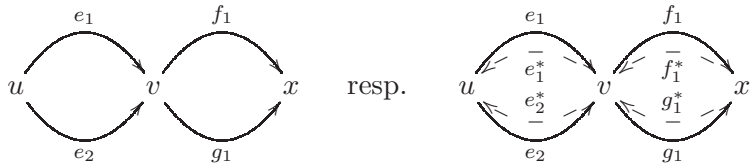
\[
\xymatrix@C+15pt{
u\ar@/^1.7pc/[r]^{e_1}\ar@/_1.7pc/[r]_{e_2}&  v\ar@/^1.7pc/[r]^{f_1}\ar@/_1.7pc/[r]_{g_1}  & x 
}\quad\text{ resp. }\quad \xymatrix@C+15pt{
u\ar@/^1.7pc/[r]^{e_1}\ar@/_1.7pc/[r]_{e_2}&  v\ar@{-->}@/_1.1pc/[l]^{e_1^*}\ar@{-->}@/^1.1pc/[l]_{e_2^*}\ar@/^1.7pc/[r]^{f_1}\ar@/_1.7pc/[r]_{g_1}  & x \ar@{-->}@/_1.1pc/[l]^{f_1^*}\ar@{-->}@/^1.1pc/[l]_{g_1^*} 
}
\]
(for ghost edges we draw dashed arrows).
One checks easily that $p:=e_2f_1g_1^*e_2^*$ and $q:=e_2f_1g_1^*e_1^*$ are quasi-cycles independent of the choice of $e^v$. Further $pqp$ and $qpq$ are nod-paths and therefore $p$ and $q$ are selfconnected. This example shows that a quasi-cycle can meet a vertex more than once.
\end{example}

The following lemma shows, that quasi-cycles behave like cycles in a way (one cannot "take a shortcut").
\begin{lemma}\label{lemshc}
Let $p=x_1\dots x_n$ be a quasi-cycle and $1\leq i,j\leq n$. Then $x_ix_j$ is a nod-path iff $i<n$ and $j=i+1$ or $i=n$ and $j=1$.
\end{lemma}
\begin{proof}
If $i<n$ and $j=i+1$ or $i=n$ and $j=1$, then clearly $x_ix_j$ is a nod-path. Suppose now that $x_ix_j$ is a nod-path. \\
\\
\underline{case 1} Suppose $i=j$. Assume that $n>1$. Then we get the contradiction that $x_i$ is a nod$^2$-path which is a subword of $p^2$ of length $1<|p|=n$. Hence $n=1$ and we have $i=j=1=n$.\\
\\
\underline{case 2} Suppose $i<j$. Then $x_j\dots x_nx_1\dots x_i$ is a nod$^2$-path which is a subword of $p^2$ of length $n-j+1+i$. It follows that $j=i+1$.\\
\\
\underline{case 3} Suppose $j<i$. Then $x_j\dots x_i$ is a nod$^2$-path which is a subword of $p^2$ of length $i-j+1$. It follows that $j=1$ and $i=n$.
\end{proof}

\begin{lemma}\label{lemexp}
If there is a selfconnected quasi-cycle $p$, then $L_K(E,w)$ has exponential growth.
\end{lemma}
\begin{proof}
Let $o$ be a nod-path such that $p\!\!\not| \,o$ and $pop$ is a nod-path. Let $n\in \N$. Consider the nod-paths
\begin{equation}
p^{i_1}op^{i_2}\dots op^{i_k}
\end{equation}
where $k,i_1,\dots,i_k\in \N$ satisfy
\begin{equation}
(i_1+\dots+i_k)|p|+(k-1)|r|\leq n.
\end{equation}
Let $A=(k,i_1,\dots,i_k)$ and $B=(k',i'_1,\dots,i'_{k'})$ be different solutions of (2). Assume that $A$ and $B$ define the same nod-path in (1). After cutting out the common beginning, we can assume that the nod-path defined by $A$ starts with $o$ and the nod-path defined by $B$ with $p$ or vice versa. Since $p\!\!\not| \,o$, it follows that $|p|>|o|$. Write $p=x_1\dots x_m$. Since the next letter after an $o$ must be a $p$, we get $x_{|o|+1}=x_1$ which contradicts Remark \ref{remnod2}(a). Hence different solutions of (2) define different nod-paths in (1). By Theorem \ref{thmhp} the nod-paths in (1) are linearly independent in $V^n$. The number of solutions of (2) is $\sim 2^n$ and hence $L_R(E,w)$ has exponential growth. 
\end{proof}

\begin{corollary}\label{corexp}
If there is a vertex $v$ and structured edges $e,f \in s^{-1}(v)$ such that $w(\alpha), w(\beta)\geq 2$, then $L_R(E,w)$ has exponential growth. 
\end{corollary}
\begin{proof}
Choose $e^v=e$. First suppose that $r(f)=v$. Then $p:=f_2$ is a quasi-cycle. Further $f_2f_2^*f_2$ is a nod-path and therefore $p$ is selfconnected. Hence, by the previous lemma, $L_R(E,w)$ has exponential growth. Now suppose that $r(f)\neq v$. Then $p:=f_2f_2^*$ is a quasi-cycle. Further $f_2f_2^*f_2f_1^*f_2f_2^*$ is a nod-path and therefore $p$ is selfconnected. Hence, by the previous lemma, $L_R(E,w)$ has exponential growth.
\end{proof}

Let $E'$ denote the set of all real and ghost edges which do not appear in a quasi-cycle. Let $P'$ denote the set of all nod-paths which are composed from elements of $E'$.
\begin{lemma}\label{lemfnt}
$|P'|<\infty$.
\end{lemma}
\begin{proof}
Let $p'=x_1\dots x_n\in P'$. Assume that there are $1\leq i <j\leq n$ such that $x_i=x_j$. Then $x_i\dots x_{j-1}$ is a nod$^2$-path. Since for any nod$^2$-path $q$ which is not a quasi-cycle there is a shorter nod$^2$-path $q'$ such that any letter of $q'$ already appears in $q$, we get a contradiction. Hence the $x_i$'s are pairwise distinct. It follows that $|P'|<\infty$ since $|E'|<\infty$.
\end{proof}
A sequence $p_1,\dots,p_k$ of quasi-cycles such that $p_i\not\approx p_j$ for any $i\neq j$ is called a {\it chain of length $k$} if $p_1 \Longrightarrow p_2\Longrightarrow \dots \Longrightarrow p_k$.
\begin{theorem}\label{thmm}
$~$\vspace{-0.2cm}
\begin{enumerate}[(i)]
\item $L_K(E,w)$ has polynomial growth iff there is no selfconnected quasi-cycle.
\item If $L_K(E,w)$ has polynomial growth, then $\GKdim L_K(E,w)=d$ where $d$ is the maximal length of a chain of quasi-cycles.
\end{enumerate}
\end{theorem}
\begin{proof}
If there is a selfconnected quasi-cycle, then $L_K(E,w)$ has exponential growth by Lemma \ref{lemexp}. Suppose now that there is no selfconnected quasi-cycle. By Theorem \ref{thmhp} the nod-paths of length $\leq n$ form a basis for $V^n$. 
Clearly we can write any nod-path of length $\leq n$ in the form 
\begin{equation}
o_1p_1^{l_1}q_1o_2p_2^{l_2}q_2o_3\dots o_kp_k^{l_k}q_ko_{k+1}
\end{equation}
where $o_i\in P'~(1\leq i \leq k+1)$, $p_1,\dots,p_k$ is a chain of quasi-cycles, $l_i\geq 0~(1\leq i \leq k)$ and $q_i\neq p_i$ is a nod-path such that $q_i|p_i~(1\leq i \leq k)$ (we allow the $o_i$'s and $q_i$'s to be the empty word). Clearly $l_1|p_1| +\dots+l_k|p_k| \leq n$. This implies that for a fixed chain $p_1,\dots,p_k$ of quasi-cycles, the number of the words in (3) is $\preccurlyeq n^k \leq n^d$. Since there are only finitely many quasi-cycles, the number of nod-paths of length $\leq n$ is $\preccurlyeq n^d$.\\
On the other hand, choose a chain $p_1,\dots,p_d$ of length $d$. Then $p_1o_1p_2\dots o_{d-1}p_d$ is a nod-path for some $o_1,\dots,o_{d-1}$ such that for any $i\in\{1,\dots,d-1\}$, $o_i$ is either the empty word or a nod-path such that $p_i\!\!\not| \,o_i$. Consider the nod-paths
\begin{equation}
p_1^{l_1}o_1p_2^{l_2}\dots o_{d-1}p_d^{l_d}
\end{equation}
where $l_1,\dots,l_d\in \N$ satisfy
\begin{equation}
l_1|p_1| +\dots+l_d|p_d|+|o_1|+\dots+|o_{d-1}|\leq n.
\end{equation}
Let $A=(l_1,\dots,l_d)$ and $B=(l'_1,\dots,l'_{d})$ be different solutions of (5). Assume that $A$ and $B$ define the same nod-path in (4). After cutting out the common beginning, we can assume that the nod-path defined by $A$ starts with $o_ip_{i+1}$ for some $i\in\{1,\dots,d-1\}$ and the nod-path defined by $B$ with $p_io_i$ or $p_i^2$. If $o_i$ is the empty word, then we get the contradiction $p_i=p_{i+1}$, since $p_i$ and $p_{i+1}$ are quasi-cycles. Suppose now that $o_i$ is not the empty word. Since $p_i\!\!\not| \,o_i$, it follows that $|o_i|<|p_i|$. Further $|p_i|<|o_i|+|p_{i+1}|$ (otherwise $p_{i+1}$ would be a subword of $p_i$ of length $<|p_i|$). Write $p_i=x_1\dots x_k$ and $p_{i+1}=y_1\dots y_m$.\\
\\
\underline{case 1} Assume that $|p_i|\leq |p_{i+1}|$. Then $o_i=x_1\dots x_j$ and $p_{i+1}=x_{j+1}\dots x_k x_1\dots x_jy_{k+1}\dots y_m$ for some $j\in\{1,\dots,k-1\}$. By Remark \ref{remnod2}(b), $x_{j+1}\dots x_k x_1\dots x_j$ is a quasi-cycle. It follows that $k=m$. Hence we get the contradiction $p_i\approx p_{i+1}$.\\
\\
\underline{case 2} Assume that $|p_i|>|p_{i+1}|$. Then $o_i=x_1\dots x_j$ and $p_{i+1}=x_{j+1}\dots x_k x_1\dots x_l$ for some $j\in\{1,\dots,k-1\}$ and $l\in\{1,\dots, j-1\}$. But this yields the contradiction that $p_{i+1}$ is a subword of $p_i^2$ of length $<|p_i|$.\\
\\
Hence different solutions of (5) define different nod-paths in (4). By Theorem \ref{thmhp} the nod-paths in (4) are linearly independent in $V^n$. The number of solutions of (5) is $\sim n^d$ and thus $n^d \preccurlyeq$ the number of nod-paths of length $\leq n$.
\end{proof}
As a corollary we recover Theorem 5 of \cite{zel12}. We use the following terminology: Let $E$ be a directed graph. We denote the set of all cyclic paths by $\CP$. If $p\in \CP$, we denote by $E(p)$ the subgraph of $E$ defined by $p$. A {\it cycle} (in the sense of \cite{zel12}) is a subgraph $E(p)$ where $p\in \CP$.
\begin{corollary}\label{corm}
Let $K$ be a field and $E$ be a finite directed graph. Then: 
\begin{enumerate}[(i)]
\item $L_K(E)$ has polynomial growth iff two distinct cycles do not have a common vertex.
\item If $L_K(E)$ has polynomial growth, then $\GKdim L_K(E)=\max(2d_1-1,2d_2)$ where $d_1$ is the maximal length of a chain of cycles and $d_2$ is the maximal length of a chain of cycles with an exit.
\end{enumerate}
\end{corollary}
\begin{proof}
Let $(E,w)$ be the weighted graph such that $w \equiv 1$. Then $L_K(E)\cong L_K(E,w)$ (see Example \ref{exex1}). It is easy to see that $\{p,p^*\mid p\in \CP\}$ is the set of all quasi-cycles of $(E,w)$ (we identify $E$ with the directed graph $\hat E$ associated to $(E,w)$). We will show (i) first and then (ii).\\\\
(i) First suppose that $L_K(E)$ has polynomial growth. Then, by the previous theorem, there is no selfconnected quasi-cycle. Assume that there are two distinct cycles $C_1$ and $C_2$ with a common vertex $v$. Then there are $p,q\in \CP$ such that $E(p)=C_1$, $E(q)=C_2$, $s(p)=r(p)=s(q)=r(q)=v$ and $p\!\!\not| \, q$. By the previous paragraph, $p$ is a quasi-cycle. Further $p$ is selfconnected (since $pqp$ is a nod-path) and hence we arrived at a contradiction. Thus two distinct cycles do not have a common vertex.\\
Now suppose that two distinct cycles do not have a common vertex. Assume that there is a selfconnected quasi-cycle $p$. We only consider the case that $p\in \CP$, the case that $p^*\in \CP$ is similar. Let $o=x_1\dots x_n$ be a path such that $pop$ is a path and $p\!\!\not| \,o$. Since $o$ is a closed path, there is a subword $p'=x_i\dots x_j$ of $o$ (where $1\leq i\leq j\leq n$) such that $p'\in \CP$. It follows from \cite[Lemma 4]{zel12} that $E(p')=E(p)$. Assume that $i=1$. Then $s(p')=r(p)=s(p)$ which implies that $p'=p$. But that contradicts $p\!\!\not| \,o$. Hence $i>1$. Set $o':=x_1\dots x_{i-1}$. Then clearly $p\!\!\not| \,o'$, $po'p'$ is a path and $|o'|<|o|$. We see that we arrive at a contradiction after repeating this step a finite number of times. Hence there is no selfconnected quasi-cycle and thus, by the previous theorem, $L_K(E)$ has polynomial growth.\\
\\
(ii) Let $d$ be the maximal length of a chain of quasi-cycles, $d_1$ the maximal length of a chain of cycles and $d_2$ is the maximal length of a chain of cycles with an exit. We have to show that $d=\max(2d_1-1,2d_2)$. Let $C_1,\dots,C_{d_2}$ a chain of cycles with an exit. Choose $p_1, \dots , p_{d_2}\in \CP$ such that $E(p_i)=C_i$ for any $1\leq i\leq d_2$. Then $p_1,\dots,p_{d_2},p^*_{d_2}, \dots, p_1^*$ is a chain of quasi-cycles (see the proof of \cite[Theorem 5]{zel12}) and hence $d\geq 2d_2$. Let now $C_1,\dots,C_{d_1}$ be a chain of cycles. If it has an exit, then $d_1=d_2$ and we have $d\geq 2d_2=2d_1>2d_1-1$. Suppose now that $C_1,\dots,C_{d_1}$ has no exit. Choose $p_1, \dots , p_{d_1}\in \CP$ such that $E(p_i)=C_i$ for any $1\leq i\leq d_1$. Then $p_1,\dots,p_{d_1},p^*_{d_1-1}, \dots, p_1^*$ is a chain of quasi-cycles (see the proof of \cite[Theorem 5]{zel12}) and hence $d\geq 2d_1-1$. Thus we have shown that $d\geq \max(2d_1-1,2d_2)$. On the other hand it is easy to see that $d\leq \max(2d_1-1,2d_2)$ and hence we have $d=\max(2d_1-1,2d_2)$ as desired.
\end{proof}

\section{Examples}
Throughout this section $K$ denotes a field. We will consider only connected, irreducible weighted graphs (cf. \cite[Definitions 21,23]{hazrat-preusser}). While wLpas of reducible weighted graphs are isomorphic to Lpas (cf. \cite[Proposition 28]{hazrat-preusser}), it is an open question which of the wLpas of irreducible weighted graphs are isomorphic to Lpas.

In general it is not so easy to read off the quasi-cycles from a finite weigthed graph. But there is the following algorithm to find all the quasi-cycles: For any vertex $v$ list all the d-paths $x_1\dots x_n$ starting and ending at $v$ and having the property that $x_i\neq x_j$ for any $i\neq j$ (there are only finitely many of them). Now delete from that list any $p$ such that $p^2$ is not a nod-path. Next delete from the list any $p$ such that $p^2$ has a subword $q$ of length $|q|<|p|$ such that $q^2$ is a nod-path. The remaining d-paths on the list are precisely the quasi-cycles starting (and ending) at $v$.

First we consider two trivial examples which show that small changes in the weighted graph can change the GK dimension of its wLpa drastically.
\begin{example}\label{ex4.0}
Consider the weighted graph 
\[
(E,w):\quad\xymatrix@C+15pt{u&v\ar[l]_{~e,2}\ar[r]^{f}&x}.
\]
One checks easily that the length of a nod-path is bounded (the unique longest nod-path is $e_2^*f_1f_1^*e_2$). Hence there is no nod$^2$-path and therefore no quasi-cycle. Thus, by Theorem \ref{thmm}, $\GKdim L_K(E,w)=0$.
\end{example}
\begin{example}
Consider the weighted graph 
\[
(E,w):\quad\xymatrix@C+15pt{u&v\ar[l]_{~e,2}\ar[r]^{f,2}&x}.
\]
By Corollary \ref{corexp}, $L_K(E,w)$ has exponential growth and hence $\GKdim L_K(E,w)=\infty$.
\end{example}

Through the next example we obtain the GK dimensions of the Leavitt algebras $L_K(n,n+k)$.
\begin{example}
Let $n\geq 1$ and $k\geq 0$. Consider the weighted graph
\[
(E,w):\quad\xymatrix{
v \ar@{.}@(l,d) \ar@(ur,dr)^{e^{(1)},n} \ar@(r,d)^{e^{(2)},n} \ar@(dr,dl)^{e^{(3)},n} \ar@(l,u)^{e^{(n+k)},n}
}.
\]
As mentioned in Example \ref{wlpapp}, $L_K(E,w)$ is isomorphic to the Leavitt algebra $L_K(n,n+k)$. If $n>1$, then, by Corollary \ref{corexp}, $L_K(E,w)$ has exponential growth and hence $\GKdim L_K(E,w)=\infty$. If $n=1$ and $k=0$, then $\GKdim L_K(E,w)=1$ by Corollary \ref{corm}. If $n=1$ and $k>0$, then  $\GKdim L_K(E,w)=\infty$ by Corollary \ref{corm}.
\end{example}

Next we consider again the weighted graph from Example \ref{exqu}.
\begin{example}
Consider the weighted graph 
\[
\xymatrix@C+15pt{
u\ar[r]^{e,2}&  v\ar@/^1.7pc/[r]^{f}\ar@/_1.7pc/[r]_{g}  & x  
}.
\]
Then $p=e_2f_1g_1^*e_2^*$ is a selfconnected quasi-cycle (see Example \ref{exqu}) and hence, by Lemma \ref{lemexp}, $L_K(E,w)$ has exponential growth. Thus $\GKdim L_K(E,w)=\infty$. 
\end{example}

The next example shows, that for any positive integer $n$ there is a connected, irreducible weighted graph $(E,w)$ such that $\GKdim L_K(E,w)=n$.
\begin{example}
Let $n\in \N$. Consider the weighted graph 
\[
(E,w):\quad\xymatrix@C+15pt{
u&v\ar[l]_{e,2}\ar[r]^{f^{(1)}}& x_1\ar@(ul,ur)^{g^{(1)}}\ar[r]^{f^{(2)}}&x_2\ar@(ul,ur)^{g^{(2)}}\ar[r]^{f^{(3)}}&\dots \ar[r]^{f^{(n)}}&x_n\ar@(ul,ur)^{g^{(n)}}}.
\]
One checks easily that the only quasi-cycles are $p_i:=g^{(i)}_1$ and $p_i^*= (g^{(i)}_1)^*~ (1\leq i\leq n)$ and that they are not selfconnected. The longest chains of quasi-cycles are $p_1,\dots,p_{n-1},p_n, p_{n-1}^*,\dots,p_1^*$ and $p_1,\dots,p_{n-1},p_n^*, p_{n-1}^*,$ $\dots,p_1^*$. Hence $\GKdim L_K(E,w)= 2n-1$ by Theorem \ref{thmm}. Consider now the weighted graph
\[
(E,w):\quad\xymatrix@C+15pt{
u&v\ar[l]_{e,2}\ar[r]^{f^{(1)}}& x_1\ar@(ul,ur)^{g^{(1)}}\ar[r]^{f^{(2)}}&x_2\ar@(ul,ur)^{g^{(2)}}\ar[r]^{f^{(3)}}&\dots \ar[r]^{f^{(n)}}&x_n\ar@(ul,ur)^{g^{(n)}}\ar[r]^{f^{(n+1)}}&x_{n+1}}.
\]
One checks easily that the only quasi-cycles are $p_i:=g^{(i)}_1$ and  $p_i^*= (g^{(i)}_1)^*~ (1\leq i\leq n)$ and that they are not selfconnected. The longest chain of quasi-cycles is $p_1,\dots,p_n, p_{n}^*,\dots,p_1^*$ and hence $\GKdim L_K(E,w)= 2n$ by Theorem \ref{thmm}.
\end{example}

We finish this section with a last example. We use the following terminology: Let $x,y\in \hat E_d^1$. If $xy$ is a nod-path, then $y$ is called a {\it nod-successor of $x$}. Let $p=x_1\dots x_n$ be a quasi-cycle, $y\in \hat E_d^1$ and $1\leq i\leq n$. If $y$ is a nod-successor of $x_i$ which is not equal to $x_{i+1}$ if $i<n$ resp. to $x_{1}$ if $i=n$, then the nod-path $x_iy$ is a called a {\it nod-exit of $p$}. Let $[p]$ denote the $\approx$-equivalence class of $p$. If $q\in [p]$, then clearly $p$ and $q$ have the same nod-exits. Hence we can define a {\it nod-exit of $[p]$} to be a nod-exit of $p$.
\begin{example}
Consider the weighted graph 
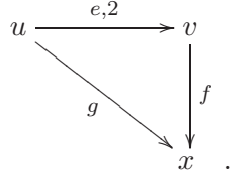
\[
\xymatrix@C+15pt@R+15pt{
u\ar[rd]_{g}\ar[r]^{e,2}&  v\ar[d]^{f}  \\& \quad x\quad.  
}
\]
Applying the algorithm described in the second paragraph of this section one gets that the only $\approx$-equivalence classes of quasi-cycles are $[e_2f_1g_1^*]$ and $[g_1f_1^*e_2^*]$.
While $[e_2f_1g_1^*]$ has no nod-exit, $[g_1f_1^*e_2^*]$ has the nod-exits 
\[e_2^*e_1, ~e_2^*e_2,~f_1^*e_1^*,~ g_1g_1^*.\]
The only nod-successor of $e_1$ is $f_1$ and the only nod-successor of $e_1^*$ and $g_1^*$ is $e_2$. But $f_1$ and $e_2$ belong to $e_2f_1g_1^*$ which has no nod-exit. We leave it to the reader to conclude that there is no selfconnected quasi-cycle and the maximal length of a chain of quasi-cycles is $2$. Thus, by Theorem \ref{thmm}, we have $\GKdim L_K(E,w)=2$.
\end{example}.

\section{Finite-dimensional weighted Leavitt path algebras}
Throughout this section $K$ denotes a field and $(E,w)$ a weighted graph. We call a $K$-algebra {\it finite-dimensional} if it is finite-dimensional as a $K$-vector space. The goal of this section is to prove that if $L_K(E,w)$ is finite-dimensional, then it is isomorphic to a finite product of matrix rings over $K$.

We call $(E,w)$ {\it aquasicyclic} if there is no quasi-cycle.
\begin{lemma}\label{5.0}
$L_K(E,w)$ is finite-dimensional iff $(E,w)$ is finite and aquasicyclic.
\end{lemma}
\begin{proof}
Follows from Theorems \ref{thmhp} and \ref{thmm} and the fact that a finitely generated $K$-algebra $A$ is finite-dimensional iff $\GKdim A=0$.
\end{proof}

Until right before the Finite Dimension Theorem \ref{thmm2}, $(E,w)$ is assumed to be finite and aquasicyclic. 
\begin{definition}[{\sc Tree}]
If $u,v\in E^0$ and there is a path $p$ in $E$ such that $s(p)=u$ and $r(p)=v$, then we write $u\geq v$. Clearly $\geq$ is a preorder on $E^0$. If $u\in E^0$ then $T(u):=\{v\in E^0 \ |  \ u\geq v\}$ is called the {\it tree of $u$}. If $X\subseteq E^0$, we define $T(X):=\bigcup\limits_{v\in X}T(v)$.
\end{definition}
\begin{definition}[{\sc Range weight forest}]\label{defn2}
A structured edge $e\in E^{1}$ is called {\it weighted} if $w(e)>1$. The subset of $E^1$ consisting of all weighted structured edges is denoted by $E^1_w$. The set $\RWF(E,w):=T(r(E^{1}_w))$ is called the {\it range weight forest of $(E,w)$}.
\end{definition}

We call $(E,w)$ {\it acyclic} if there is no cyclic path in $E$. We call two structured edges $e$ and $f$ {\it in line} if $e=f$ or $r(e)\geq s(f)$ or $r(e)\geq s(f)$.
\begin{lemma}\label{5.1}
$(E,w)$ is acyclic, any vertex $v$ emits at most one weighted structured edge, any vertex $v\in \RWF(E,w)$  emits at most one structured edge and $T(r(e))\cap T(r(f))=\emptyset$ for any $e, f \in E^{1}_w$ which are not in line.
\end{lemma}
\begin{proof}
That $(E,w)$ is acyclic is clear since any cyclic path in $\hat E$ is a quasi-cycle. Assume there is a vertex $v$ which emits two distinct weighted structured edges. The proof of Corollary \ref{corexp} shows that then there is a quasi-cycle and hence we have a contradiction.\\ 
Suppose now that there is a $v\in \RWF(E,w)$ such that $|s^{-1}(v)|\geq 2$. Choose an $e\in s^{-1}(v)\setminus \{e^v\}$. By the definition of $\RWF(E,w)$, there is a structured edge $f\in E^{1}_w$ and a path $p$ in $\hat E$ such that $\hat s(p)=r(f)$ and $\hat r(p)=v$. Then $f_2pe_1e_1^*p^*f_2^*$ (resp. $f_2e_1e_1^*f_2^*$ if $|p|=0$) is a nod$^2$-path. Since for any nod$^2$-path $q$ which is not a quasi-cycle there is a nod$^2$-path $q'$ such that $|q'|<|q|$, the existence of a quasi-cycle follows.\\
Now let $e, f\in E^{1}_w$ be not in line. Assume that there is a $v\in T(r(e))\cap T(r(f))$. Then there are paths $p=x_1\dots x_m$ and $q=y_1\dots y_n$ in $\hat E$ such that $\hat s(p)=r(e)$, $\hat s(q)=r(f)$ and $\hat r(p)=\hat r(q)=v$. W.l.o.g. assume that $m\geq n$. After cutting off a possible common ending of $p$ and $q$ we may assume that we are in one of the following three cases:\\
\\
\underline{case 1} Assume that $|p|,|q|\neq 0$ and $x_m\neq y_n$. Then $e_2pq^*f_2^*f_2qp^*e_2^*$ is a nod$^2$-path and hence there is a quasi-cycle.\\
\\
\underline{case 2} Assume that $|p|\neq 0$ and $|q|=0$. Clearly $x_m\neq f_i$ for any $1\leq i\leq w(f)$ since $e$ and $f$ are not in line. Hence $e_2pf_2^*f_2p^*e_2^*$ is a nod$^2$-path and hence there is a quasi-cycle.\\
\\
\underline{case 3} Assume that $|p|=|q|=0$. Since $e$ and $f$ are distinct, $e_2f_2^*f_2e_2^*$ is a nod$^2$-path and hence there is a quasi-cycle.
\end{proof}

We call an $e\in E^{1}_w$ {\it weighted structured edge of type A} if $s(e)$ emits only one structured edge (namely $e$) and {\it weighted structured edge of type B} otherwise. 
\begin{lemma}\label{5.2}
There is a finite, aquasicyclic weighted graph $(\tilde E,\tilde w)$ such that $L_K(\tilde E,\tilde w)\cong L_K(E,w)$, $(\tilde E,\tilde w)$ has at most as many weighted structured edges as $(E, w)$, all weighted structured edges in $(\tilde E, \tilde w)$ are of type $B$ and their ranges are sinks.
\end{lemma}
\begin{proof}
Set $Z:=\RWF(E,w)\cup\{v\in E^0\mid s^{-1}(v)=\{e\} \text{ for some }e\in E^{1}_w\}$.
Define a weighted graph $(\tilde E,\tilde w)$ by $\tilde E^0=E^0$, $\tilde E^1=\{e\mid e\in E^1,s(e)\not\in Z\}\cup \{e^{(1)},\dots,e^{w(e)}\mid e\in E^1,s(e)\in Z\}$, $\tilde s(e)=s(e)$, $\tilde r(e)=r(e)$ and $\tilde w(e)=w(e)$ if $s(e)\not\in Z$ and $\tilde s(e^{(i)})=r(e)$, $\tilde r(e^{(i)})=s(e)$ and $\tilde w(e^{(i)})=1$ if $s(e)\in Z$ and $1\leq i\leq w(e)$. One checks easily that $(\tilde E,\tilde w)$ has at most as many weighted structured edges as $(E, w)$, all weighted structured edges in $(\tilde E, \tilde w)$ are of type $B$ and their ranges are sinks. The proof that $L_K(\tilde E,\tilde w)\cong L_K(E,w)$ is very similar to the proof of \cite[Proposition 28]{hazrat-preusser} and therefore is omitted. That $(\tilde E,\tilde w)$ is finite and aquasicyclic follows from Lemma \ref{5.0}.
\end{proof}

\begin{example}\label{ex5.00}
Suppose $(E,w)$ is the finite, aquasicyclic weighted graph 
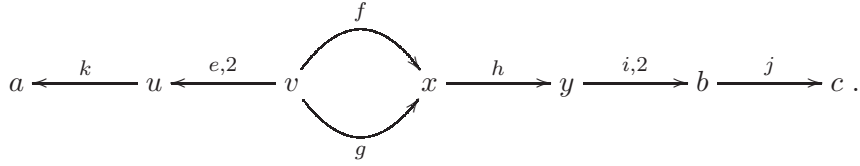
\[
\xymatrix@C+15pt{a&u\ar[l]_{k}&v\ar[l]_{e,2}\ar@/^1.7pc/[r]^{f}\ar@/_1.7pc/[r]_{g}&x\ar[r]^{h}&y\ar[r]^{i,2}&b\ar[r]^{j}&c}.
\]
The weighted structured edge $i$ is of type A and the weighted structured edge $e$ is of type B. Clearly $\RWF(E,w)=\{a,u,b,c\}$. Let $Z$ be defined as in the proof of the previous lemma. Then $Z=\{a,u,y,b,c\}$. Let $(\tilde E, \tilde w)$ be the weighted graph
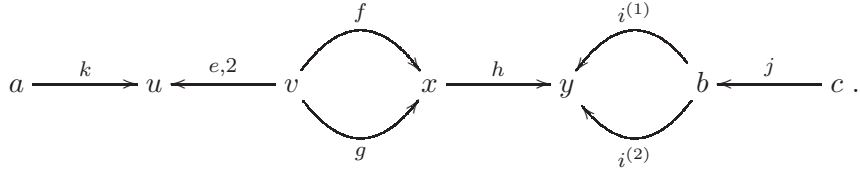
\[
\xymatrix@C+15pt{ a \ar[r]^{k}& u& v\ar[l]_{e,2}\ar@/^1.7pc/[r]^{f}\ar@/_1.7pc/[r]_{g}& x\ar[r]^{h}& y& b\ar@/^1.7pc/[l]^{i^{(2)}}\ar@/_1.7pc/[l]_{i^{(1)}}& c\ar[l]_{j}}.
\]
Then $L_K(E,w)\cong L_K(\tilde E, \tilde w)$. In $(\tilde E,\tilde w)$ there is only one weighted structured edge, namely $e$, and it is of type $B$. Further $u=\tilde r(e)$ is a sink.
\end{example}

The next goal is to remove also the weighted structured edges of type B without changing the wLpa, so that eventually one arrives at an unweighted graph (i.e. at a weighted graph $(E',w')$ such that $w'\equiv 1$). Recall that a subset $H\subseteq E^0$ is called {\it hereditary} if $u\geq v$ where $u\in H$ and $v\in E^0$ implies $v\in H$. 
\begin{definition}[{\sc Weighted subgraph defined by hereditary vertex set}]
Let $H\subseteq E^0$ be a hereditary subset. Set $E_H^0:=H$, $E_H^1:=\{e\in E^1\mid s(e)\in H\}$, $r_H:=r|_{E_H^1}$, $s_H=s|_{E_H^1}$ and $w_H:=w|_{E_H^1}$. Then $E_H:=(E_H^0,E_H^1,s_H,r_H)$ is a directed graph and $(E_H,w_H)$ a weighted graph. We call $(E_H,w_H)$ the {\it weighted subgraph of $(E,w)$ defined by $H$}.   
\end{definition}

\begin{lemma}\label{5.3}
Suppose that all weighted structured edges in $(E, w)$ are of type $B$ and their ranges are sinks. Let $v\in E^0$ be a vertex such that $v$ is the only element of $T(v)$ which emits a weighted structured edge. Then there is an unweighted graph $(E',w')$ such that $L_K(E_{T(v)}, w_{T(v)})\cong L_K(E',w')$ via an isomorphism which maps vertices to sums of distinct vertices.
\end{lemma}
\begin{proof}
Clearly there is an integer $k\geq 2$, integers $m, n_1,\dots,n_m\geq 1$, a vertex $u\in E^0\setminus\{v\}$, pairwise distinct vertices $x_1,\dots,x_m\in E^0\setminus\{v\}$ and pairwise distinct structured edges $e, f^{(ij)}\in E^1~(1\leq i \leq m, 1\leq j\leq n_i)$ such that $s^{-1}(v)=\{e, f^{(ij)}\mid 1\leq i \leq m, 1\leq j\leq n_i\}$, $r(e)=u$, $r(f^{(ij)})=x_i$, $w(e)=k$ and $w(f^{(ij)})=1$. Set $X:=\{x_1,\dots,x_m\}$. It is easy to see that $u,v\not\in T(X)$ (otherwise there would be a quasi-cycle). Define an unweighted graph $(E',w')$ by 
\begin{align*}
(E')^0:=&\{u_i~(1\leq i\leq k),u_{ij}~(1\leq i\leq m, 1\leq j \leq (k-1)n_i),\\
&v, v_{ij}~(1\leq i\leq m, 2\leq j \leq n_i),x_i~(1\leq i \leq m), y~(y\in T(X)\setminus X)\},\\
(E')^1:=&\{\alpha^{(i)}~(1\leq i\leq k),\beta^{(ij)}~(1\leq i\leq m, 1\leq j \leq n_i),\\
& \gamma^{(ij)}~(1\leq i\leq m, 1\leq j \leq (k-1)n_i), g~(g\in E^1, s(g)\in T(X))\},\\
s'(\alpha^{(1)})=&v,~r'(\alpha^{(1)})=u_1, ~s'(\alpha^{(i)})=u_{i-1},~r'(\alpha^{(i)})=u_i~(i\geq 2)\\
s'(\beta^{(i1)})=&v,~r'(\beta^{(i1)})=v_{i2},~ s'(\beta^{(ij)})=v_{ij},~r'(\beta^{(ij)})=v_{i,j+1}~(1<j<n_i),~s'(\beta^{(in_i)})=v_{in_i},~r'(\beta^{(in_i)})=x_{i},\\
s'(\gamma^{(i1)})=&x_i,~r'(\gamma^{(i1)})=u_{i1}, ~s'(\gamma^{(ij)})=u_{i,j-1},~r'(\gamma^{(ij)})=u_{ij}~(j\geq 2)\\
s'(g)=&s(g),~r'(g)=r(g)~(g\in E^1, s(g)\in T(X)\setminus X), ~s'(g)=u_{i,(k-1)n_i},~r'(g)=r(g)~(g\in E^1, s(g)=x_i).
\end{align*}
Define an algebra homomorphisms $\phi:L_K(E_{T(v)}, w_{T(v)})\rightarrow L_K(E',w')$ by 
\begin{align*}
\phi(u)&=\sum\limits_{1\leq i \leq k}u_i+\sum\limits_{\substack{1\leq i\leq m,\\1\leq j\leq (k-1)n_i}}u_{ij},\\
\phi(v)&=v+\sum\limits_{\substack{1\leq i\leq m,\\2\leq j\leq n_i}}v_{ij},\\
\phi(x_i)&=x_i~(1\leq i\leq m),\\
\phi(y)&=y~(y\in T(X)\setminus X),\\
\phi(e_1)&=\alpha_1^{(1)},~\phi(e_1^*)=(\alpha_1^{(1)})^*,\\
\phi(e_l)&=\alpha_1^{(1)}\dots \alpha_1^{(i)}+\sum\limits_{\substack{1\leq i\leq m,\\1\leq j\leq n_i}}\beta_1^{(ij)}\dots \beta_1^{(in_i)}\gamma_1^{(i1)}\dots \gamma_1^{(i,(l-2)n_i+j)},\\
\phi(e_l^*)&=(\alpha_1^{(i)})^*\dots (\alpha_1^{(1)})^*+\sum\limits_{\substack{1\leq i\leq m,\\1\leq j\leq n_i}}(\gamma_1^{(i,(l-2)n_i+j)})^*\dots (\gamma_1^{(i1)})^*(\beta_1^{(in_i)})^*\dots (\beta_1^{(ij)})^*~(2\leq l\leq m),\\
\phi(f_1^{(ij)})&=\beta_1^{(ij)}\dots \beta_1^{(in_i)},~\phi((f_1^{(ij)})^*)=(\beta_1^{(in_i)})^*\dots (\beta_1^{(ij)})^*~(1\leq i\leq m,1\leq j \leq n_i),\\
\phi(g_1)&=g_1,~\phi(g_1^*)=g_1^*~(g\in E^1, s(g)\in T(X)\setminus X),\\
\phi(g_1)&=\gamma_1^{(i1)}\dots \gamma_1^{(i,(k-1)n_i)}g_1~,\phi(g_1^*)=g_1^*(\gamma_1^{(i,(k-1)n_i)})^*\dots (\gamma_1^{(i1)})^*~(g\in E^1, s(g)=x_i)
\end{align*}
and an algebra homomorphism $\psi:L_K(E',w')\rightarrow L_K(E_{T(v)}, w_{T(v)})$ by
\begin{align*}
\psi(u_i)&=e_i^*e_1e_1^*e_i~(1\leq i\leq k),\\
\psi(u_{i,(j-2)n_i+l})&=e_j^*f_1^{(il)}(f_1^{(il)})^*e_j~(1\leq i\leq m,2\leq j \leq k,1\leq l \leq n_i),\\
\psi(v)&=e_1e_1^*+\sum\limits_{1\leq i\leq m}f_1^{(i1)}(f_1^{(i1)})^*,\\
\psi(v_{ij})&=f_1^{(ij)}(f_1^{(ij)})^*~(1\leq i\leq m,2\leq j\leq n_i),\\
\psi(x_i)&=x_i~(1\leq i\leq m),\\
\psi(y)&=y~(y\in T(X)\setminus X),\\
\psi(\alpha_1^{(1)})&=e_1,~\psi((\alpha^{(1)})^*)=e_1^*,\\
\psi(\alpha_1^{(i)})&=e_{i-1}^*e_1e_1^*e_i,~\psi((\alpha_1^{(i)})^*)=e_{i}^*e_1e_1^*e_{i-1}~(2\leq i\leq k),\\
\psi(\beta_1^{(ij)})&=f_1^{(ij)}(f_1^{(i,j+1)})^*,~\psi((\beta_1^{(ij)})^*)=f_1^{(i,j+1)}(f_1^{(ij)})^*~(1\leq i\leq m, 1\leq j <n_i),\\
\psi(\beta_1^{(in_i)})&=f_1^{(in_i)},~\psi((\beta_1^{(in_i)})^*)=(f_1^{(in_i)})^*~(1\leq i\leq m),\\
\psi(\gamma_1^{(i1)})&=(f^{(i1)}_1)^*e_2,~\psi((\gamma^{(i1)})^*)=e_2^*f^{(i1)}_1~(1\leq i\leq m),\\
\psi(\gamma_1^{(i(j-2)n_i+1)})&=e_{j-1}^*f^{(in_i)}_1(f^{(i1)}_1)^*e_j,~\psi((\gamma_1^{(i,(j-2)n_i+l)})^*)=e_j^*f^{(i1)}_1(f^{(in_i)}_1)^*e_{j-1},~(1\leq i\leq m,3\leq j \leq k),\\
\psi(\gamma_1^{(i,(j-2)n_i+l)})&=e_j^*f^{(i,l-1)}_1(f^{(il)}_1)^*e_j,~\psi((\gamma_1^{(i,(j-2)n_i+l)})^*)=e_j^*f^{(il)}_1(f^{(i,l-1)}_1)^*e_j,~(1\leq i\leq m,2\leq j \leq k,\\
&\hspace{12.7cm} 2\leq l \leq n_i),\\
\psi(g_1)&=g_1,~\psi(g_1^*)=g_1^*~(g\in (E')^1, s'(g)\in T(X)\setminus X),\\
\psi(g_1)&=e^*_kf_1^{(in_i)}g_1,~\psi(g_1^*)=g_1^*(f_1^{(in_i)})^*e_k~(g\in (E')^1, s'(g)=u_{i,(k-1)n_i}).
\end{align*}
It follows from the universal properties of $L_K(E_{T(v)}, w_{T(v)})$ and $L_K(E', w')$ that $\phi$ and $\psi$ are well defined. One checks easily that $\phi\circ \psi=id_{L_K(E',w')}$ and $\psi\circ \phi=id_{L_K(E_{T(v)}, w_{T(v)})}$. Thus $L_K(E_{T(v)}, w_{T(v)})\cong L_K(E', w')$.
\end{proof}
\begin{example}\label{ex5.0}
Suppose $(E,w)$ is the weighted graph 
\[
\xymatrix@C+15pt{ u& v\ar[l]_{e,2}\ar[r]^{f}& x}
\]
from Example \ref{ex4.0}.
Clearly $(E,w)=(E_{T(v)}, w_{T(v)})$. Let $(E',w')$ be the unweighted graph
\[
\xymatrix@C+15pt{ u_2&u_1\ar[l]_{\alpha^{(2)}}&v\ar[l]_{\alpha^{(1)}}\ar[r]^{\beta^{(11)}}&x\ar[r]^{\gamma^{(11)}}&u_{11}}.
\]
Then $L_K(E, w)\cong L_K(E',w')$ by the previous lemma.
\end{example}
\begin{example}\label{ex5.1}
Suppose $(E,w)$ is the weighted graph 
\[
\xymatrix@C+15pt{ a \ar[r]^{k}& u& v\ar[l]_{e,2}\ar@/^1.7pc/[r]^{f}\ar@/_1.7pc/[r]_{g}& x\ar[r]^{h}& y& b\ar@/^1.7pc/[l]^{i^{(2)}}\ar@/_1.7pc/[l]_{i^{(1)}}& c\ar[l]_{j}}.
\]
Then $(E_{T(v)}, w_{T(v)})$ is the weighted graph
\[
\xymatrix@C+15pt{u& v\ar[l]_{e,2}\ar@/^1.7pc/[r]^{f}\ar@/_1.7pc/[r]_{g}& x\ar[r]^{h}& y}.
\]
Let $(E',w')$ be the unweighted graph
\[
\xymatrix@C+15pt{ u_2&u_1\ar[l]_{\alpha^{(2)}}&v\ar[l]_{\alpha^{(1)}}\ar[r]^{\beta^{(11)}}&v_{12}\ar[r]^{\beta^{(12)}}&x\ar[r]^{\gamma^{(11)}}&u_{11}\ar[r]^{\gamma^{(12)}}&u_{12}\ar[r]^{h}&y}.
\]
Then $L_K(E_{T(v)}, w_{T(v)})\cong L_K(E',w')$ by the previous lemma..
\end{example}

We want to show that we can "replace" the subgraph $(E_{T(v)}, w_{T(v)})$ in the previous lemma by the unweighted graph $(E', w')$ within $(E,w)$ without changing the wLpa. In order to do that we need some definitions.
\begin{definition}[{\sc Weighted graph homomorphism}]
Let $(\tilde E,\tilde w)$ be a weighted graph. A morphism $f: (E,w) \rightarrow (\tilde E,\tilde w)$ consists of maps $f^0: E^0 \rightarrow \tilde E^0$ and $f^{1}: E^{1} \rightarrow \tilde E^{1}$ such that $\tilde r(f^{1}(e)) = f^0(r(e))$, $\tilde s(f^{1}(e)) = f^0(s(e))$ and $\tilde w(f^{1}(e))=w(e)$ for any $e \in E^{1}$.
\end{definition}
\begin{definition}[{\sc Complete weighted subgraph}]
A {\it weighted subgraph of $(E,w)$} is a weighted graph $(\tilde E,\tilde w)$ where $\tilde E^0\subseteq E^0$, $\tilde E^1\subseteq E^1$, $\tilde s=s|_{\tilde E^1}$, $\tilde r=r|_{\tilde E^1}$ and $\tilde w=w|_{\tilde E^1}$. A weighted subgraph $(\tilde E,\tilde w)$ of $(E,w)$ is called {\it complete} if $\tilde s^{-1}(v)=s^{-1}(v)$ for any $v\in \tilde E^0_{\reg}$.
\end{definition}
\begin{lemma}\label{lememb}
Let $(\tilde E,\tilde w)$ denote a complete weighted subgraph of $(E,w)$. Then the canonical graph monomorphism $(\tilde E,\tilde w)\rightarrow (E,w)$ induces an algebra monomorphism $L_K(\tilde E,\tilde w)\rightarrow L_K(E,w)$.
\end{lemma}
\begin{proof}
The existence of an algebra homomorphism $L_K(\tilde E,\tilde w)\rightarrow L_K(E,w)$ follows from the universal property of $L_K(E,w)$. That it is injective follows from Theorem \ref{thmhp} since nod-paths in $(\tilde E,\tilde w)$ are mapped to nod-paths in $(E,w)$.
\end{proof}
\begin{definition}[{\sc Replacement graph}]
Let $H\subseteq E^0$ a hereditary subset, $(E',w')$ be a weighted graph and $\phi:L_K(E_H,w_H)\rightarrow L_K(E',w')$ an isomorphism which maps vertices to sums of distinct vertices, i.e. for any $v\in H$ there are distinct $u'_{v,1},\dots,u'_{v,n_v}\in (E')^0$ such that $\phi(v)=u'_{v,1}+\dots+u'_{v,n_v}$. The weighted graph $(\tilde E, \tilde w)$ defined by   
\begin{align*}
&\tilde E^0=E^0\setminus H\sqcup (E')^0,\\
&\tilde E^{1}=\{e\mid e\in E^{1}, s(e),r(e)\in E^0\setminus H\}\\
&\quad\quad\sqcup\{e^{(1)},\dots,e^{(n_{r(\alpha)})}\mid e\in E^{1}, s(e)\in E^0\setminus H,r(e)\in H\}\\
&\quad\quad\sqcup (E')^{1},\\
&\tilde s(e)=s(e),\tilde r(e)=r(e),\tilde w(e)= w(e)\quad(e\in E^{1}, s(e),r(e)\in E^0\setminus H),\\
&\tilde s(e^{(j)})=s(e),\tilde r(e^{(j)})=u'_{r(\alpha),j},\tilde w(e^{(j)})= w(e)\quad(e\in E^{1}, s(e)\in E^0\setminus H,r(e)\in H ,1\leq j\leq n_{r(e)}),\\
&\tilde s(e')=s'(e'),\tilde r(e')=r'(e'),\tilde w(e')=w'(e')\quad(e'\in (E')^{1})
\end{align*}
is called the {\it replacement graph defined by $\phi$}.
\end{definition}
\begin{replacement lemma}\label{5.4}
Let $H\subseteq E^0$ be a hereditary subset, $(E',w')$ a weighted graph and $\phi:L_K(E_H,w_H)\rightarrow L_K(E',w')$ an isomorphism which maps vertices to sums of distinct vertices. Then $L_K(E,w)\cong L_K(\tilde E, \tilde w)$ where $(\tilde E, \tilde w)$ is the replacement graph defined by $\phi$.
\end{replacement lemma}
\begin{proof}
Clearly $(E',w')$ is a complete weighted subgraph of $(\tilde E, \tilde w)$. By Lemma \ref{lememb}, there is an algebra monomorphism $\psi: L_K(E',w')\rightarrow L_K(\tilde E, \tilde w)$. Define an algebra homomorphisms $f:L_K(E,w)\rightarrow L_K(\tilde E,\tilde w)$ by 
\begin{align*}
f(v)&=v\quad(v\in E^0\setminus H),\\
f(v)&=\psi(\phi(v))\quad(v\in H),\\
f(e_i)&=e_i,f(e_i^*)=e_i^*\quad(e\in E^{1}, s(e),r(e)\in E^0\setminus H, 1\leq i\leq w(e)),\\
f(e_i)&=\sum\limits_{j=1}^{n_{r(e)}} e^{(j)}_i,f(e_i^*)=\sum\limits_{j=1}^{n_{r(e)}}( e^{(j)}_i)^*\quad(e\in E^{1}, s(e)\in E^0\setminus H,r(e)\in H, 1\leq i\leq w(e)),\\
f(e_i)&=\psi(\phi(e_i)),f(e_i^*)=\psi(\phi(e^*_i))\quad(e\in E^{1}, s(e),r(e)\in H, 1\leq i\leq w(e))
\end{align*}
and an algebra homomorphism $g:L_K(\tilde E,\tilde w)\rightarrow L_K(E,w)$ by
\begin{align*}
g( v)&=v\quad(v\in E^0\setminus H),\\
g( v')&=\phi^{-1}(v')\quad(v'\in (E')^0),\\
g( e_i)&= e_i,g(e_i^*)= e^*_i\quad(e\in E^{1}, s(e),r(e)\in E^0\setminus H, 1\leq i\leq w(e)),\\
g( e_i^{(j)})&=e_i\phi^{-1}(u'_{r(e),j}),g(( e_i^{(j)})^*)=\phi^{-1}(u'_{r(e),j})e^*_i\quad(e\in E^{1}, s(e)\in E^0\setminus H,r(e)\in H,
\\&\hspace{8cm} 1\leq i\leq w(e)),1\leq j\leq n_{r(e)}),\\
g( e'_i)&=\phi^{-1}(e'_i),g(( e'_i)^*)=\phi^{-1}((e'_i)^*)\quad(e'\in (E')^{1}, 1\leq i\leq w'(e')).
\end{align*}
It follows from the universal properties of $L_K(E,w)$ and $L_K(\tilde E, \tilde w)$ that $f$ and $g$ are well defined. One checks easily that $f\circ g=id_{L_K(\tilde E, \tilde w)}$ and $g\circ f=id_{L_K(E,w)}$. Thus $L_K(E,w)\cong L_K(\tilde E, \tilde w)$.
\end{proof}
\begin{example}\label{ex5.2}
Suppose $(E,w)$ is the weighted graph 
\[
\xymatrix@C+15pt{ a \ar[r]^{k}& u& v\ar[l]_{e,2}\ar@/^1.7pc/[r]^{f}\ar@/_1.7pc/[r]_{g}& x\ar[r]^{h}& y& b\ar@/^1.7pc/[l]^{i^{(2)}}\ar@/_1.7pc/[l]_{i^{(1)}}& c\ar[l]_{j}}.
\]
Let $(E',w')$ be the unweighted graph
\[
\xymatrix@C+15pt{u_2&u_1\ar[l]_{\alpha^{(2)}}&v\ar[l]_{\alpha^{(1)}}\ar[r]^{\beta^{(11)}}&v_{12}\ar[r]^{\beta^{(12)}}&x\ar[r]^{\gamma^{(11)}}&u_{11}\ar[r]^{\gamma^{(12)}}&u_{12}\ar[r]^{h}&y}.
\]
Then, as mentioned in Example \ref{ex5.1}, $L_K(E_{T(v)}, w_{T(v)})\cong L_K(E',w')$. Let $\phi$ be the isomorphism defined in the proof of Lemma \ref{5.3}. Then the replacement graph defined by $\phi$ is the unweighted graph 
\[
(\tilde E, \tilde w):\quad\xymatrix@C+15pt{u_2&u_1\ar[l]_{\alpha^{(2)}}&v\ar[l]_{\alpha^{(1)}}\ar[r]^{\beta^{(11)}}&v_{12}\ar[r]^{\beta^{(12)}}&x\ar[r]^{\gamma^{(11)}}&u_{11}\ar[r]^{\gamma^{(12)}}&u_{12}\ar[r]^{h}&y\\
&&&a\ar[lllu]^{k^{(1)}}\ar[llu]_{k^{(2)}}\ar[rru]^{k^{(3)}}\ar[rrru]_{k^{(4)}}&&&&b\ar@/^1.7pc/[u]^{i^{(2)}}\ar@/_1.7pc/[u]_{i^{(1)}}\\
&&&&&&&c\ar[u]_{j}}.
\]
By the previous lemma we have $L_K(E,w)\cong L_K(\tilde E, \tilde w)$.
\end{example}

Now we are ready to prove the main result of this section.
\begin{FD theorem}\label{thmm2}
Let $K$ denote a field and $(E,w)$ a weighted graph. Then the following statements are equivalent.
\begin{enumerate}[(i)]
\item $L_K(E,w)$ is finite-dimensional.
\item $(E,w)$ is finite and aquasicyclic.
\item $L_K(E,w)\cong \prod\limits_{i=1}^{m}M_{n_i}(K)$ for some $m,n_1,\dots,n_m\in\N$.
\end{enumerate}
\begin{proof}
(i)$\Leftrightarrow$(ii). Holds by Lemma \ref{5.0}.\\
(ii)$\Rightarrow$(iii). Suppose that $(E,w)$ is finite and aquasicyclic (and hence $L_K(E,w)$ is finite-dimensional by Lemma \ref{5.0}). By Lemma \ref{5.2} we may assume that all weighted structured edges in $(E, w)$ are of type $B$ and their ranges are sinks. Consider the vertices in $E^0$ which emit weighted structured edges. It is easy to see that at least one of them, say $v$, has the property that $v$ is the only element of $T(v)$ which emits a weighted structured edge (otherwise there would be a cyclic path). By Lemma \ref{5.3} there is an unweighted graph $(E',w')$ such that $L_K(E_{T(v)}, w_{T(v)})\cong L_K(E',w')$ via an isomorphism $\phi$ which maps vertices to sums of distinct vertices. By the Replacement Lemma \ref{5.4}, $L_K(E,w)\cong L_K(\tilde E, \tilde w)$ where $(\tilde E, \tilde w)$ is the replacement graph defined by $\phi$. Clearly $(\tilde E, \tilde w)$ has one weighted structured edge less than $(E,w)$. We see that after a finite number of applications of Lemmas \ref{5.2}, \ref{5.3} and \ref{5.4} we arrive at an unweighted graph $(E'',w'')$ such that $L_K(E'',w'')\cong L_K(E,w)$. Since $L_K(E,w)$ is finite-dimensional, $L_K(E'',w'')$ is finite-dimensional. It follows from \cite[Theorem 2.6.17]{abrams-ara-molina} that $L_K(E,w)$ is isomorphic to a finite product of matrix rings over $K$.\\
(iii)$\Rightarrow$(i). Clear.
\end{proof}
\end{FD theorem}
\begin{example}
Suppose $(E,w)$ is the weighted graph 
\[
\xymatrix@C+15pt{ u& v\ar[l]_{e,2}\ar[r]^{f}& x}
\]
from Example \ref{ex4.0} and Example \ref{ex5.0}. Let $(E',w')$ be the unweighted graph
\[
\xymatrix@C+15pt{ u_2&u_1\ar[l]_{\alpha^{(2)}}&v\ar[l]_{\alpha^{(1)}}\ar[r]^{\beta^{(11)}}&x\ar[r]^{\gamma^{(11)}}&u_{11}}.
\]
Then $L_K(E, w)\cong L_K(E',w')$ as mentioned in Example \ref{ex5.0}. It follows from \cite[Theorem 2.6.17]{abrams-ara-molina} that $L_K(E,w)\cong M_3(K)\times M_3(K)$.
\end{example}
\begin{example}
Suppose $(E,w)$ is the weighted graph 
\[
\xymatrix@C+15pt{a&u\ar[l]_{k}&v\ar[l]_{e,2}\ar@/^1.7pc/[r]^{f}\ar@/_1.7pc/[r]_{g}&x\ar[r]^{h}&y\ar[r]^{i,2}&b\ar[r]^{j}&c}.
\]
from Example \ref{ex5.00}. Let $(\tilde E, \tilde w)$ be the unweighted graph
\[
(\tilde E, \tilde w):\quad\xymatrix@C+15pt{u_2&u_1\ar[l]_{\alpha^{(2)}}&v\ar[l]_{\alpha^{(1)}}\ar[r]^{\beta^{(11)}}&v_{12}\ar[r]^{\beta^{(12)}}&x\ar[r]^{\gamma^{(11)}}&u_{11}\ar[r]^{\gamma^{(12)}}&u_{12}\ar[r]^{h}&y\\
&&&a\ar[lllu]^{k^{(1)}}\ar[llu]_{k^{(2)}}\ar[rru]^{k^{(3)}}\ar[rrru]_{k^{(4)}}&&&&b\ar@/^1.7pc/[u]^{i^{(2)}}\ar@/_1.7pc/[u]_{i^{(1)}}\\
&&&&&&&c\ar[u]_{j}}.
\]
Then $L_K(E,w)\cong L_K(\tilde E, \tilde w)$ by Examples \ref{ex5.00} and \ref{ex5.2}. It follows from \cite[Theorem 2.6.17]{abrams-ara-molina} that $L_K(E,w)\cong M_5(K)\times M_{12}(K)$.
\end{example}


\begin{thebibliography}{99}

\bibitem{aap05} G. Abrams, G. Aranda Pino, \emph{The Leavitt path algebra of a graph}, J. Algebra {\bf 293} (2005), no. 2, 319--334.

\bibitem{abrams-ara-molina} G. Abrams, P. Ara, M. Siles Molina, Leavitt path algebras, Lecture Notes in Mathematics {\bf 2191}, Springer, 2017.

\bibitem{zel12}  A. Alahmadi, H. Alsulami, S. Jain, E. Zelmanov, \emph{Leavitt path algebras of finite Gelfand-Kirillov dimension},  J. Algebra Appl. {\bf 11} (2012), no. 6, 1250225.



\bibitem{aragoodearl} P. Ara, K. Goodearl, \emph{Leavitt path algebras of separated graphs}, J. reine angew. Math. {\bf 669} (2012), 165--224.

\bibitem{Ara_Moreno_Pardo} P. Ara, M.A. Moreno, E. Pardo, \emph{Nonstable $K$-theory for graph algebras}, Algebr. Represent. Theory {\bf 10} (2007), no. 2, 157--178.





\bibitem{cuntz1} J. Cuntz,  \emph{Simple $C^*$-algebras generated by isometries}, Comm. Math. Phys. {\bf 57} (1977), no. 2, 173--185.


\bibitem{hazrat13} R. Hazrat, \emph{The graded structure of Leavitt path algebras}, Israel J. Math. {\bf 195} (2013), no. 2, 833--895. 


\bibitem{hazrat-preusser} R. Hazrat, R. Preusser, \emph{Applications of normal forms for weighted Leavitt path algebras: simple rings and domains}, Algebr. Represent. Theor. {\bf 20} (2017), 1061–-1083. 



\bibitem{vitt56} W.G. Leavitt, \emph{Modules over rings of words}, Proc. Amer. Math. Soc. {\bf 7} (1956), 188--193. 

\bibitem{vitt57} W.G. Leavitt, \emph{Modules without invariant basis number,} Proc. Amer. Math. Soc. {\bf 8} (1957), 322--328.

\bibitem{vitt62} W.G.  Leavitt, \emph{The module type of a ring,} Trans. Amer. Math. Soc. {\bf 103} (1962) 113--130. 





\bibitem{phillips} N.C. Phillips, \emph{A classification theorem for nuclear purely infinite simple $C^*$-algebras},
Doc. Math. {\bf 5} (2000), 49--114.

\bibitem{raeburn} I. Raeburn, Graph algebras. CBMS Regional Conference Series in Mathematics, {\bf 103}
American Mathematical Society, Providence, RI, 2005.



\end{thebibliography}
\end{document}